\documentclass[]{article}
\usepackage{amsmath}
\usepackage{amsfonts}
\usepackage{amssymb}
\usepackage[english]{babel}
\usepackage{graphicx}
\usepackage{amsthm}
\usepackage{hyperref}
\theoremstyle{plain}

\newtheorem*{theorem*}{Theorem}
\newtheorem{theorem}{Theorem}
\theoremstyle{definition}
\newtheorem{definition}{Definition}[section]
\theoremstyle{lemma}
\newtheorem{lemma}{Lemma}[section]
\theoremstyle{corollary}
\newtheorem{corollary}{Corollary}[section]
\theoremstyle{claim}
\newtheorem{claim}{Claim}[section]
\theoremstyle{proposition}
\newtheorem{proposition}{Proposition}[section]

\theoremstyle{remark}

\begin{document}
\global\long\def\l{\lambda}
\global\long\def\ep{\epsilon}

\title{Outer Billiards with Contraction: Regular Polygons}
\author{In-Jee Jeong}

\maketitle

\begin{abstract}
	We study outer billiards with contraction outside regular polygons. For regular $n$-gons with $n = 3, 4, 5, 6, 8$, and $12$, we show that as the contraction rate approaches $1$, dynamics of the system converges, in a certain sense, to that of the usual outer billiards map. These are precisely the values of $n \geq 3$ with $[\mathbb{Q}(e^{2\pi i/n}):\mathbb{Q}] \leq 2$. Then we discuss how such convergence may fail in the case of $n=7$.
	
\end{abstract}
\section{Introduction}

Polygonal outer billiard is a fascinating problem. This system exhibits very diverse
behavior for different polygons. While there are many interesting results in this area,
several important problems are open. For one thing, it is not known if every orbit remains bounded
for generic convex polygons. The class of convex polygons whose orbits are known to be bounded
is nowhere dense in the set of convex polygons (see \cite{MR1145593,MR2991430}).
On the other hand, `irrational kites' are the only known examples of having unbounded orbits \cite{MR2854095,MR2562898}. Second, we know little about the structure of the set of periodic points (which comes as a union of polygonal `tiles') for generic convex polygons.

\begin{figure}
	\begin{centering}
		\includegraphics[scale=0.5]{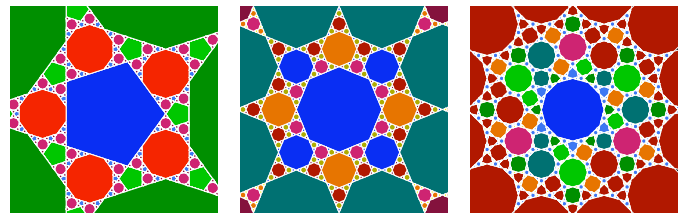}
		\par\end{centering}
	\caption{The Outer Billiards Map; Picture by R. Schwartz}\label{fig:cat1}
\end{figure}

These problems in polygonal outer billiards are difficult due to the lack of general methods for attacking them. For this reason, most of the research in polygonal outer billiards have
been conducted by case studies; the case of regular pentagon was analyzed in \cite{MR1354670},
trapezoids in \cite{MR2708027}, kites in \cite{MR2854095,MR2562898}, `semiregular' octagons 
in \cite{1006.2782}, a few regular polygons in \cite{MR2835332}. 

In this paper, we study the outer billiards map outside regular polygons, but now composed with an affine contraction. For each convex polygon, we get a one-parameter family of dynamical systems parametrized by the contraction rate $0<\l \leq 1$. We mostly focus on the case $\l \approx 1$, and we consider it as a perturbation of the usual outer billiards map. Our main motivation for applying contraction is to take the limit $\l \nearrow 1$ and study whether the asymptotic dynamics is related in any sense to that of the usual outer billiards system. To this end, we introduce two notions that are related with each other; $\l$-stability and convergence of picture. The first notion asks whether a periodic point of the outer billiards map persists for values of contraction close to $1$. The second one asks if there is a convergence of the partition of the plane according to the $\omega$-limit set as the contraction rate approaches $1$. In the course of establishing $\l$-stability for certain periodic points, we will discover some symmetry of the outer billiards map which is not so apparent otherwise (see Corollary \ref{cor:lattice}).

Outer billiards with contraction was studied as well in  \cite{Jeong2014,GianluigiDelMagno2013}, which discuss several other motivations for its study. The paper \cite{GianluigiDelMagno2013} contains many beautiful pictures of the system.

We review some basic facts regarding regular polygons. For our purposes, it will be convenient to divide regular $n$-gons into three categories:

\begin{itemize}
	\item Category I: $n=3, 4,$ and $6$. These are precisely the values of $n$ where regular $n$-gon is a lattice polygon, up to an affine transformation of the plane. For these, the outer billiard map is `trivial' as all orbits are periodic.
	
	\item Category II: $n=5, 8$, and $12$. Together with Category I, these are precisely the values where $\phi(n) \leq 2$ where $\phi$ is the Euler totient function. For them, the outer billiards map is `completely renormalizable'; some of its consequences are density of periodic orbits and self-similar fractal structure of the set of non-periodic points. 
	
	\item Category III: All other values of $n \geq 3$. 
\end{itemize}

Not much is known for the dynamics outside regular polygons in Category III. Two very interesting open problems are whether periodic tiles are dense in the domain and whether there exists infinitely many non-similar tiles. We will come back to these questions later. 

Let us state the main results of this paper. To begin with, for Category I polygons we have a good understanding of the dynamics:

\begin{claim}Let $P$ be a regular polygon from Category I. Then for any value of the contraction $0<\l<1$, there exists finitely many periodic orbits to which all other orbits are attracted. As we take the limit $\l \nearrow  1$, we recover all periodic orbits of the (usual) outer billiards map outside $P$. 
\end{claim}\label{claim1}

For Category I, II, we have: 

\begin{claim}Let $P$ be a regular polygon from Category I, II. Then any periodic orbit of the outer billiards map for $P$ is $\l$-stable, and we have convergence of picture.
\end{claim}\label{claim2}

This notion of $\l$-stability may distinguish Category III polygons from others:

\begin{claim}Let $P$ be the regular septagon. Then there exists a periodic tile which is not $\l$-stable and convergence of picture fails. 
\end{claim}\label{claim3}

The structure of this paper is as follows. In Section \ref{Pre}, we collect basic facts about the outer billiards map and define the terms that are used throughout the paper, in particular the terms that are previously mentioned. In addition, we introduce a $\l$-stability criterion which will be used to establish aforementioned claims. Section \ref{sec:main} has three subsections and we address each claim in each subsection. 

\section{Preliminaries}\label{Pre}

\subsection{Basic Definitions}

We start by defining our systems.

\begin{definition}[The System]
	\label{system}
	Given a pair $(P,\l)$ of a convex polygon $P$ and a number $0<\l<1$, we define the outer billiards with contraction $T_\l$ as follows. For a generic point $x\in\mathbb{R}^{2}\backslash P$,
	we can find a unique vertex $v$ of $P$ such that $P$ lies on the
	left side of the ray starting from $x$ and passing through $v$. On this ray, we pick the point $y$ which lies on the opposite side of $x$ with respect to $v$ and satisfies $|xv|:|vy|=1:\l$. Then we define $T_\l x=y$ (Figure \ref{fig:The-outer-billiard}). The map $T_\l$ is well-defined for all points on $\mathbb{R}^{2}\backslash P$ except for points on the union of singular rays extending the sides of $P$. If we denote this singular set by $S$, then we have a well-defined map on the domain $X:=\mathbb{R}^2 \backslash \big(P\cup (\cup_{i=0}^{\infty}T_\l ^{-i}S))$.
\end{definition}

\begin{figure}
	\begin{centering}
		\includegraphics{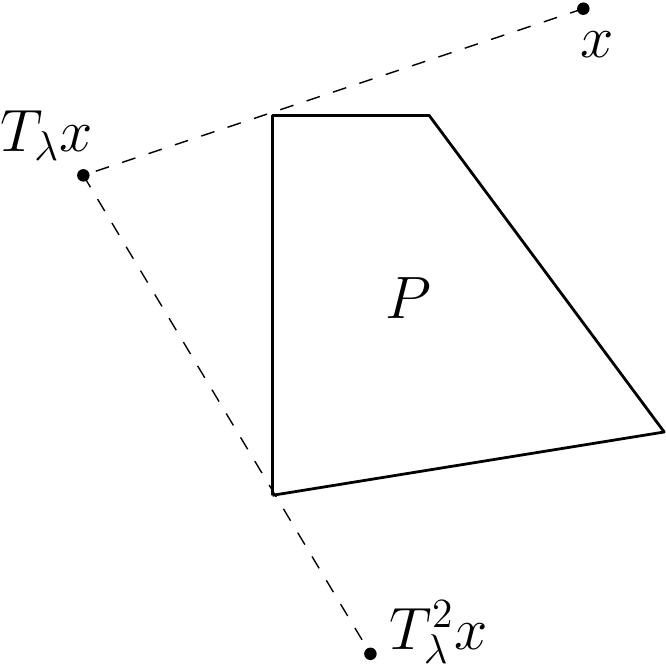}
		\par\end{centering}
	\caption{Outer billiards with contraction}\label{fig:The-outer-billiard}
\end{figure}

Note that the domain $X$ depends on the polygon as well as the contraction rate $\l$. The case $\l =1 $ corresponds to the usual outer billiards, and we denote the map simply by $T$. 

One observes that the dynamics of outer billiards, either with or without contraction is invariant under an orientation preserving affine transformation of the plane. Indeed, if two convex polygons are mapped to each other by such a transformation, the same transformation acts as a conjugacy between maps $T_\l$ for these polygons. 

We say that a polygon is lattice if its every vertex has integer coordinates. A polygon is called affine-lattice if it can be mapped to a lattice polygon via an orientation preserving affine transformation of the plane. The following result proves the statement in the introduction regarding regular polygons in Category I, as they are the only affine-lattice regular polygons.

\begin{theorem*}[\cite{MR1145593}]
	Every point is periodic for $T$ when $P$ is a convex affine-lattice polygon.
\end{theorem*}

Now our goal consists of defining the notion of $\l$-stability.

\begin{definition}[Symbolic Coding]
	Assume that a convex polygon $P$ and a value $0<\l \leq 1$ is given. Label the vertices of $P$ by $1, 2,...,n$ where $n$ is the number of vertices of $P$. Then for each point $x$ in the domain $X$, we associate an infinite sequence of integers (the code of $x$) $\{a_k(x)\}_{k=0}^{\infty}$ where $a_k$ is the label of the vertex $T_\l^k (x)$ gets reflected on. 
	Now for each $x$, we consider the set of points in $X$ which share the same code with $x$; this is a convex set. We denote it by $\mathrm{SCR}_\l(x)$ (same combinatorial region).
\end{definition}

\begin{lemma}[Periodic Tile] \label{lem:tile}
	Let $x$ be a periodic point for $T$ outside a polygon $P$. Then the maximal connected domain of periodic points containing $x$ is an (open) convex polygon. We call this polygon the periodic tile (for $x$). 
	
	On the other hand, if $x$ is a periodic point for $T_\l$ for some $0<\l<1$, then the maximal connected domain of points $y$ with the property that $d(T_\l ^i (y),T_\l ^i (x)) \rightarrow 0$ as $i \rightarrow \infty$ is again an (open) convex polygon.
\end{lemma}

\begin{proof}
	Consider the periodic part of code of $x$, say, $\{a_0,...,a_{k-1}\}$. Define $\Psi = T_{k-1} \circ ... \circ T_0$ where $T_i$ is the reflection across the vertex of $P$ corresponding to $a_i$. Then $\Psi$ is a piecewise isometry, and therefore it must be the identity map (we can assume $k$ is even by repeating the code once more if necessary) on the piece containing $x$. 
	The second statement can be proved analogously. 
\end{proof}

Note that for a $T$-periodic point $x$, the periodic tile containig $x$ is simply $\mathrm{SCR}_1(x)$. 

\begin{lemma} \label{uniqueness}
	For a convex polygon $P$ whose vertices are labeled by $1,2,...,n$, take any finite sequence $C=\{a_0,...,a_{k-1} \}$ where each $a_i \in \{1,2,...,n\}$. For each $0<\l<1$, there exists at most one periodic point whose code is the repetition of $C$. This periodic point, if exists, is given explicitly by the formula 
	\begin{equation}
		q_C(\l)=\dfrac{1-(-\lambda)}{1-(-\lambda)^{k}}(\sum_{i=0}^{k-1}(-\lambda)^{k-1-i}v_{i}).\label{eq:periodic_point}
	\end{equation}
	where $v_i \in \mathbb{R}^2$ is the coordinate of the vertex of $P$ with label $a_i$.
\end{lemma}

\begin{proof}
	Consider the map $\Phi:=F_{k-1} \circ ... \circ F_0$ where $F_i$ is the reflection with respect to $v_i$ composed with the contraction by $\l$. Then $\Phi$ is a contractive map of the plane, so it has unique fixed point. This point may or may not be a valid periodic point for $T_\l$. The formula follows since $\Phi(p)=(-\l)^k p+(1+\l)(\sum_{i=0}^{k-1}(-\l)^{k-1-i}v_i)$. 
\end{proof}

\begin{definition}[$\l$-stability]
	Let $x$ be a periodic point for $T$ outside some polygon $P$. We say that the point $x$ (alternatively, the periodic tile containing $x$) is $\l$-stable if there exists $\ep>0$ such that for all $1-\ep<\l<1$, there exists a periodic point for $T_\l$ which has the same code with $x$. 
\end{definition}

That is, a periodic tile for $T$ is stable in this sense if it gives rise to a periodic orbit (which is unique by Lemma \ref{uniqueness}) for $T_\l$ when we slightly decrease $\l$ from 1. It is not a trivial matter to decide if a given tile is $\l$-stable or not.

Finally, we introduce the notion of convergence of picture. The motivation is as follows. For each $n \geq 3$ and for any $0 < \l <1$, it is believed that there are only finitely many periodic orbits outside the regular $n$-gon and they all come from periodic orbits of the usual outer billiards map, which form a countable set. Therefore, we may pick a countable set of colors and associate each of them to a periodic orbit. Then for each $0 \l \leq 1$, color points of the domain according to the color associated with the periodic orbit that the point is asymptotic to (\cite{GianluigiDelMagno2013} contains many pictures of this kind). We ask if the pictures generated converge in some weak sense to the picture for the case $\l = 1$. 

\begin{definition}[Convergence of picture]\label{def:conv}
	For each point $x$ in the domain of $T$ for the regular $n$-gon, consider the sequence of sets $\{ \mathrm{SCR}_{\l}(x)  \}_{0<\l\leq 1}$. If the point $x$ does not lie on the domain of $T_\l$, just define $ \mathrm{SCR}_{\l}(x) $ by the singleton $\{x\}$. We say that there is convergence of picture if for almost every point in the domain of $T$, the sequence converges to $\mathrm{SCR}_{1}(x) $ in the Hausdorff topology.
\end{definition}

In the following section, we will prove the convergence of picture for $n = 3,4,5,6,8$, and $12$. For the proof we will use a simple observation: if (i) periodic tiles for $T$ is dense and (ii) every periodic tile is $\l$-stable, then it implies convergence of picture. The item (i) follows from the lattice structure for $n \in \{3,4,6\}$ and from the renormalization scheme for $n \in \{5,8,12\}$. The second item will follow from a $\l$-stability criterion we prove in the next subsection together with some analysis of the periodic tiles. 

\subsection{Unfolding Scheme for Outer Billiards}\label{subsec:unfolding}

The outer billiards map is often called the dual billiards map. Indeed on the sphere, there is an exact duality with inner billiards (see \cite{Tabachnikov1995}). For the inner billiards map inside a polygonal region, there is a very useful unfolding method in which instead of reflecting the particle trajectory, one reflects the polygon while keeping the trajectory straight. The unfolding scheme we are going to describe is simply the dual version of it for outer billiards. 

Let $P$ be a convex polygon and $x \in X$ be a point in the domain of $T$. Assume that $x$ reflects on the vertex $v$ of $P$. Then instead of moving the point $x$, we reflect $P$ with respect to $v$. We can simply repeat this procedure, obtaining a connected chain of copies of $P$ surrounding the point $x$. 

Clearly, a point $x \in X$ is periodic if and only if after several reflections, the polygon comes back to its starting position. We may name these copies of $P$ by $P, T(P), T^2(P)$, and so on. Note that for a point $y$ to have the same dynamics with $x$ it is necessary and sufficient that $y$ belongs to the region formed by the angle between $T^i(P)$ and $T^{i+1}(P)$ for every $i$ (e.g. the region bounded by two dashed lines in Figure \ref{fig:unfolding}). From this criterion, one sees that the tile for $x$ is just the intersection of many half-planes. Since these half-planes come in at most $2n$ slopes, we immediately deduce:

\begin{lemma}[see \cite{MR1145593},\cite{MR1354670}]
	Let $P$ be an $n$-gon. Given a periodic point $x$ for $T$, its periodic tile is a polygon with the number of sides not exceeding $2n$. 
\end{lemma}

\begin{figure}
	\begin{centering}
		\includegraphics[scale=0.6]{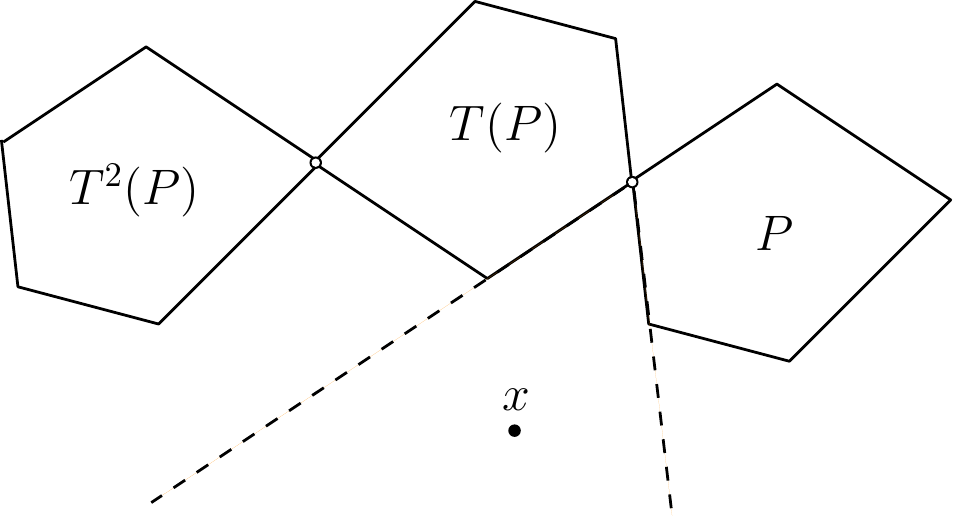}
		\par\end{centering}
	\caption{Unfolding scheme}\label{fig:unfolding}
\end{figure}

\begin{definition}[Symmetric tiles]\label{def:sym}
	Let $P$ be a regular $n$-gon centered at the origin and $Q$ be a periodic tile. Rotate $Q$ counterclockwise by integer multiples of $2\pi/n$ across the origin. The rotated images are again periodic tiles, and if some of them are obtained by $T$-iterates of $Q$, we say that the periodic tile $Q$ is symmetric. 
\end{definition}

From the previous lemma, it is clear that 

\begin{lemma}\label{lem:symmetry}
	If $Q$ is a symmetric tile, then $Q$ is rotationally symmetric. In particular, when $P$ is a regular $p$-gon for $p$ prime, a symmetric tile is necessarily either a regular $p$-gon or a regular $2p$-gon. 
\end{lemma}

This lemma explains, to some extend, abundance of periodic tiles of regular $n$- and $2n$-gons outside the regular $n$-gon. Let us now introduce a $\l$-stability criterion. 

\begin{proposition}\label{prop:criterion}
	Let $Q$ be a tile for $P$ and let $k$ be its period, which we assume to be even. Pick any point $p$ in $P$ and consider the sequence of points $p_0=p, p_1,...,p_k=p_0$ obtained by unfolding $P$ with respect to $Q$. Then $Q$ is $\l$-stable if and only if the barycenter $\sum_{i=0}^{k-1}p_i/k$ lies in the interior of $Q$. 
\end{proposition}

\begin{proof}
	Let $\{v_0,...,v_{k-1} \}$ be the sequence of vertices of $P$ corresponding to the code of $Q$. Recall that we had
	\begin{equation}
	q(\lambda)=\dfrac{1-(-\lambda)}{1-(-\lambda)^{k}}(\sum_{i=0}^{k-1}(-\lambda)^{k-1-i}v_{i}).\nonumber
	\end{equation}
	We claim that $Q$ is $\l$-stable if and only if $\lim_{\l \nearrow 1}q(\l)$ is contained in the interior of $Q$. Indeed, if the limit is contained in $Q$, there is some $\ep >0$ such that for $1-\ep < \l \leq 1$, each iterate $T^j q(\l)$ is contained in $T^j Q$, respectively. Therefore, $q(\l)$ gives a periodic orbit for $T_\l$ for $\l$ in this range. The other direction is clear.
	
	Now we do the unfolding. We will obtain a sequence of copies of $P$ surrounding $Q$; call them $P, T(P),$ and so on. Simply pick $p_0$ to be the origin; then, the vector from $p_0$ to the vertex where first reflection of $P$ occurs is simply $v_0$. In the same way, the vector from the point $p_i$ to the vertex where $i+1$th reflection occurs is $(-1)^i v_i$. 
	
	Since $Q$ is a periodic tile, we have $\sum_{i=0}^{k-1}(-\lambda)^{k-1-i}v_{i}=0$. Apply L'Hospital's rule to $q(\l)$ to obtain  
	\begin{equation}
	\lim_{\l \nearrow 1} q(\lambda)=\dfrac{2}{k} \big(\sum_{i=0}^{k-1} (k-1-i)w_i).\label{eq:st}
	\end{equation}
	
	That is, we are asking whether $V_0 := p_0 + \frac{2}{k} \big(\sum_{i=0}^{k-1} (k-1-i)w_i) \in \mathrm{int}(Q)$ or not. However, we could have started from the tile $T(Q)$. Since the code for $T(Q)$ is given by shifting that for $Q$, we see that $Q$ is $\l$-stable if and only if $V_1 := p_1 + \frac{2}{k} \big(\sum_{i=0}^{k-1} (k-1-i)w_{i+1}) \in \mathrm{int}(Q)$.
	
	Indeed, we have 
	\begin{equation} 
	\begin{split}
	V_{1}-V_{0} & = p_{1}-p_{0}+\frac{2}{k}(\sum_{i=0}^{k-1}(k-1-i)w_{i}-\sum_{i=0}^{k-1}(k-1-i)w_{i+1}) \nonumber \\
	& = p_{1}-p_{0}+\frac{2}{k}(\sum_{i=0}^{k-1}w_{i}-(k-1)w_{0}) \\
	& = p_{1}-p_{0}+\frac{2}{k}(-kw_{0})=p_{1}-p_{0}-2w_{0}=0,
	\end{split}
	\end{equation}
since $\sum_{i=0}^{k-1} w_i=0$. We define $V_j$ for each $T^j(Q)$ and similarly as above, $V_j$ is independent of $j$. Hence
	\[
	V_0 = \dfrac{1}{k}\sum_{j=0}^{k-1}V_j = \dfrac{1}{k}\big(\sum_j p_j + \sum_j \sum_i \dfrac{2}{k}(k-1-i)w_{i+j} \big) =\dfrac{1}{k}\sum_{j}p_j
	\] by an interchange of summation. 
\end{proof}

\begin{corollary}\label{cor:sym}
	Any symmetric tile $Q$ is $\l$-stable. Indeed, the curve $q(\l)$ of hypothetical $T_\l$-periodic orbits corresponding to $Q$ converges to the center of $Q$ as $\l \nearrow 1$. 
\end{corollary}

\begin{proof}
	By rotational symmetry of the unfolding picture, $\sum_i p_i=0$ where we can take $p$ to be the center of $P$.
\end{proof}

Let us mention that for $n=4$, one can explicitly calculate the sum in equations \ref{eq:periodic_point} or \ref{eq:st} and obtain the same conclusion; see appendix for details.

One may wonder how a periodic tile can be not $\l$-stable. Let $P$ be any convex polygon and $Q$ be a periodic tile outside $P$ of even period $N$. Then for $\l$ very close to 1, $T^N_{\l}(Q)$ will be a polygon similar to $Q$, converging to $Q$ as $\l \nearrow 1$. In this setting, one sees that $Q$ is $\l$-stable if and only if for $\l$ sufficiently close to 1, $T^N_{\l}(Q)$ is contained in $Q$. Indeed we are going to demonstrate a non-stable tile later. 

We prove some corollaries. First, from equation \ref{eq:st} we deduce that

\begin{corollary}\label{cor:lattice}
	Let $Q$ be a symmetric periodic tile outside a regular $n$-gon. Then the coordinates of its center lies in the field $\mathbb{Q}(e^{2\pi i/n})$. 
\end{corollary}

This will hold for \textit{all} periodic tiles in the cases $n \in \{3,4,5,6,8,12\}$. 

Next, it is not hard to see that for each regular polygon $P$, there exists an infinite sequence of necklace periodic tiles which consists of regular polygons with size identical to that of $P$ (e.g. necklace pentagons in Figure \ref{fig:necklace}). They are symmetric tiles, and we conclude: 

\begin{corollary}
	Any regular polygon has infinitely many $\l$-stable periodic tiles. In particular, the number of periodic points for $T_\l$ diverges as $\l \nearrow 1$. 
\end{corollary}

\section{Main Results}\label{sec:main}

\subsection{Case of the Square}\label{sec:square}

Recall that the map $T$ outside the square simply permutes the open squares of the same label in Figure \ref{fig:grid}, and every point inside a square of label $k$ is periodic with period $4k$. From this picture, we observe that this `label' is a Lyapunov function; that is, it does not increase on any orbits. 

\begin{figure}
	\begin{centering}
		\includegraphics{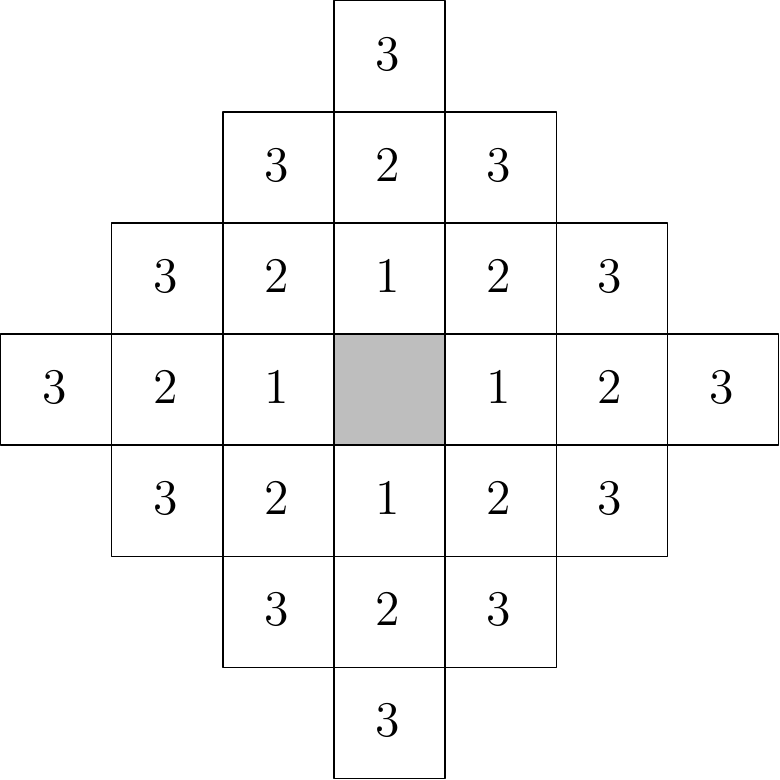}
		\par\end{centering}
	\caption{Square grid}\label{fig:grid}
\end{figure}

We will use the following simple fact (for the proof one may see the last section of \cite{Jeong2014}). 

\begin{lemma}[Boundedness of the orbits]
	\label{lem:bdd}
	Let $||\cdot||$ be any norm on the plane, and let $\{v_{1},...,v_{n}\}$ be the set of vertices of $P$. Then for any point $x$ in the domain of $T_\l$,
	\[
	\limsup_{k \rightarrow \infty }||T_{\l}^k(x)|| \leq \frac{1+\lambda}{1-\lambda}\max_{i}||v_{i}||
	\]. 
\end{lemma}

\begin{theorem}\label{thm:sq1}
	For each $0<\l<1$, and for $n=3, 4, 6$, there exists finitely many periodic orbits for $T_\l$ outside the regular $n$-gon to which all other orbits are attracted.
\end{theorem}

\begin{proof}
	Consider the case $n=4$. Notice that with any $0<\lambda<1$, $T_{\lambda}$ never increases the index of the square that the orbit of a point $p$ lies on. Since
	the index cannot decrease indefinitely, it should stabilize at some $k$. Notice that once it stabilizes, the symbolic coding of the point simply follows that of a periodic point of $T$ with index $k$. Therefore, upon iteration of $T_\l$, our orbit converges to a periodic orbit of $T_\l$ of index $k$ (which was at most unique). From Lemma \ref{lem:bdd} we know that $\limsup_k ||T^k p|| \leq C(\l)$ for any $\l<1$ where $C(\l)$ is some constant, there cannot exist infinitely many periodic orbits for $T_\l$. The same argument goes through cases $n=3, 6$. 
\end{proof}

All tiles are symmetric in the sense of Definition \ref{def:sym} and from Corollary \ref{cor:sym}, we have that all periodic tiles for regular polygons with $n=3, 4,$ and $6$ are $\l$-stable. With the observation after Definition \ref{def:conv}, we have

\begin{proposition}\label{prop:conv}
	We have convergence of picture for $n=3,4,$ and $6$.
\end{proposition}

\subsection{Case of the Regular Pentagon}\label{sec:pentagon}

For reader's convenience, we describe in some detail the behavior of the map $T$ outside a regular pentagon, following \cite{MR1354670}. We fix our $P$ in the plane to have vertices as five 5th roots of unity. Our goal consists of enumerating all periodic tiles for the regular pentagon.

See Figure \ref{fig:necklace} where two largest periodic tiles outside the regular pentagon are drawn. These periodic orbits are called necklaces as each of the orbit separates the plane into two regions.

\begin{figure}
	\begin{centering}
		\includegraphics[scale=0.5]{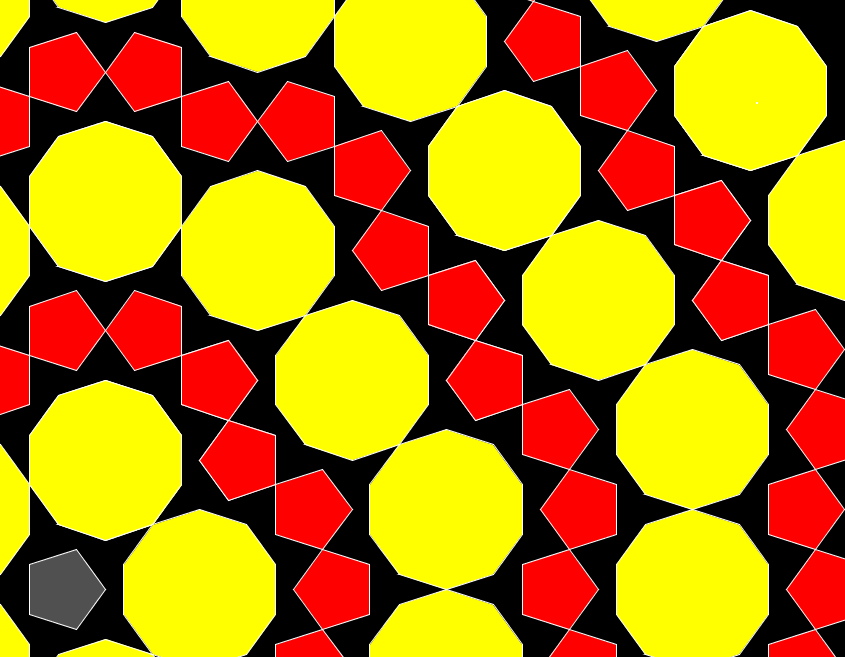}
		\par\end{centering}
	\caption{Necklace periodic orbits}\label{fig:necklace}
\end{figure}

First, pentagonal tiles in the figure come with periods 20, 40,... and have the same size with $P$. Indeed, these tiles are characterized by the property of self-duality; if we unfold $P$ around them, we simply recover their $T$-iterates on the plane. It means that $T$ acts transitively on each layer of these pentagons. Next, we see regular decagonal tiles of period 5, 15, 25,... and so on. Again $T$ acts transitively on each layer. These two sequences of necklaces divide the plane into invariant regions, say, $I_0, I_1, I_2,...$ and so on, starting from the innermost one.

Modulo 5-fold rotational symmetry, region $I_0$ is partitioned into five wedge-shaped figures. Fix one of them and call it $W$. Notice that all the other regions $I_1,I_2,...$ can be covered by (not necessarily disjoint) many copies of $W$. Indeed, when we cover each invariant region by wedges, the first-return map of $T$ is well-defined within each wedge, and this first-return map is conjugate to the first-return map on $W$. Such a statement becomes clearer in the unfolding coordinates (introduced in \cite{MR1354670}).

We now describe dynamics on the wedge $W$. We identify each point in $W$ with its images under rotation around the center by multiples of $2\pi/5$. With this identification, $T$ induces a self-map on $W$ which we still denote by $T$. 

The wedge is partitioned into two triangles $\triangle{KOL}$ and $\triangle{NLM}$ and $T$ acts as a rotation on each of them (Figure \ref{fig:wedge}). Then one can see that there is a regular decagon (with center $O_1$) and two regular pentagons which are invariant by $T$ and $T^2$, respectively. Let us call this regular decagon by $D_1$ and one of two regular pentagons by $P_1$. 

Now let $\Gamma$ be the affine map sending $KOMN$ to $K_1 O M_1 N_1$ (these two wedges are similar). It is straightforward to verify that $\Gamma$ is a renormalization map; that is, 
(1) $\Gamma$ is a conjugacy for $T$ on $W$ with the first-return map of $T$ on the small wedge $\Gamma(W)$.
(2) $T$-iterates of $\Gamma(W)$ cover $W$ modulo periodic domains $D_1, P_1,$ and $T(P_1)$.

From property (1), we find two infinite sequences of periodic tiles for $T$; $\{ \Gamma^i (D_1) \}_{i=0}^{\infty}$ and $\{ \Gamma^i (P_1) \}_{i=0}^{\infty}$. From property (2), we know that the areas of those periodic tiles together with $T$-iterates add up to the area of $W$. Therefore we have found all periodic tiles. One can further prove that the number of iterates of each periodic tile does not divide 5. Therefore, if we go back to the initial outer billiards map, $T$ must act transitively on each level of pentagonal and decagonal periodic tiles.

By conjugacy, we see that for each invariant region $I_i$, one has corresponding sequences of pentagonal and decagonal periodic tiles which densely fills out each region. With some additional effort, one can show that $T$ acts transitively on each level in these regions as well. 

\begin{figure}
	\begin{centering}
		\includegraphics[scale=1]{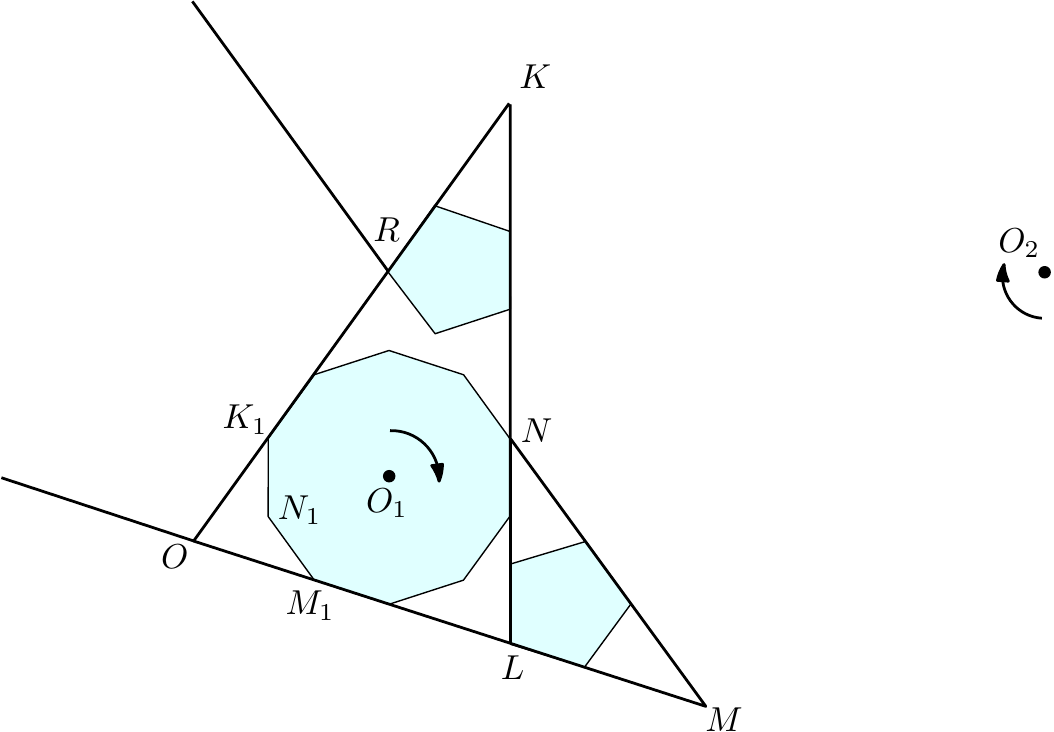}
		\par\end{centering}
	\caption{Dynamics on the wedge}\label{fig:wedge}
\end{figure}

Transitivity is important for us, as it implies symmetry of the tiles (Definition \ref{def:sym}). From Corollary \ref{cor:sym}, we have arrived at:

\begin{theorem}\label{thm:stability}
	Every periodic tile outside the regular pentagon is $\l$-stable.
\end{theorem}

A similar type of effort will prove corresponding result for the regular octagon and the regular 12-gon. Moreover, 

\begin{proposition}\label{prop:conv2}
	We have the convergence of picture for the regular pentagon.
\end{proposition}

\subsection{Additional Remarks}\label{sec:fin}

Behavior of the outer billiards map outside a regular septagon is full of mysteries. See Figure \ref{fig:septagon} which shows a series of pentagonal tiles, which have period 57848 and diameter approximately 0.0003 (the regular septagon has radius 1). These tiles were found by R. Schwartz. 

As they are not rotationally symmetric, they are not symmetric in the sense of Definition \ref{def:sym}. Using Proposition \ref{prop:criterion}, we have checked with a computer program that these pentagonal tiles are not $\l$-stable. A rigorous verification of this computer calculation (as well as the very fact that such a pentagonal tile really exists) consists of comparing two numbers in the field $\mathbb{Q}[e^{2\pi i/7}]$. Such a calculation can be executed with exact arithmetic: See an elementary algorithm by P. Hooper \cite{HOOPER}. Moreover, due to the non-stability of this tile, it is very unlikely that the convergence of picture holds for the regular septagon, even `locally'. 

\begin{figure}
	\begin{centering}
		\includegraphics[scale=0.5]{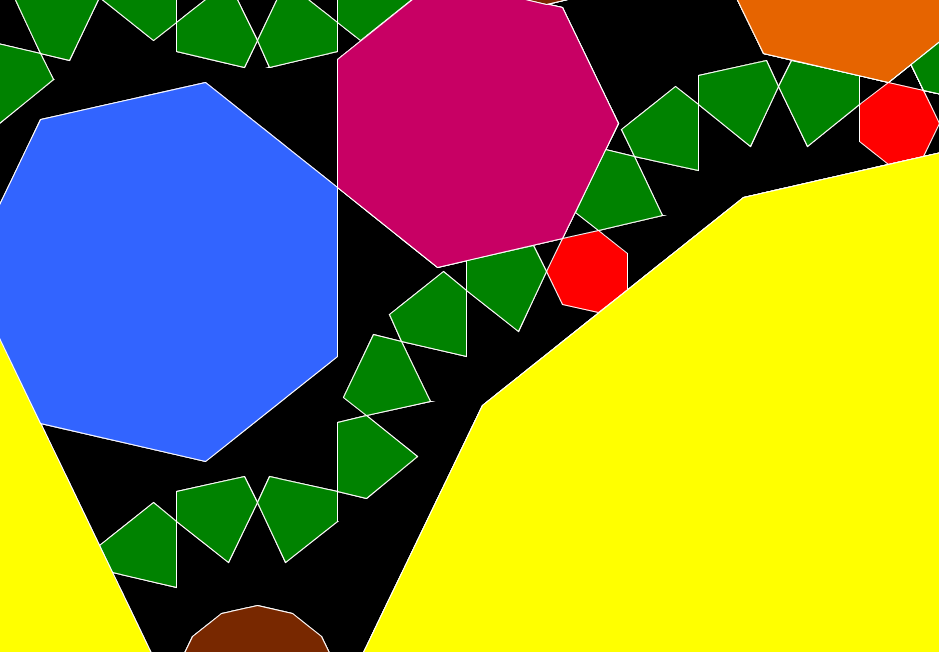}
		\par\end{centering}
	\caption{Non-symmetric pentagonal tiles}\label{fig:septagon}
\end{figure}

Indeed, R. Schwartz discovered lot more non-symmetric tiles. See Figure \ref{fig:septagon2} which shows all the periodic tiles in the region up to period 1048576. These pictures suggest that a complete renormalization scheme outside the regular septagon is unlikely to exist. One certainly needs to first come down to these very small scales ($\sim0.0001$) of `exotic' periodic tiles to search for such a scheme. 

\begin{figure}
	\begin{centering}
		\includegraphics[scale=0.5]{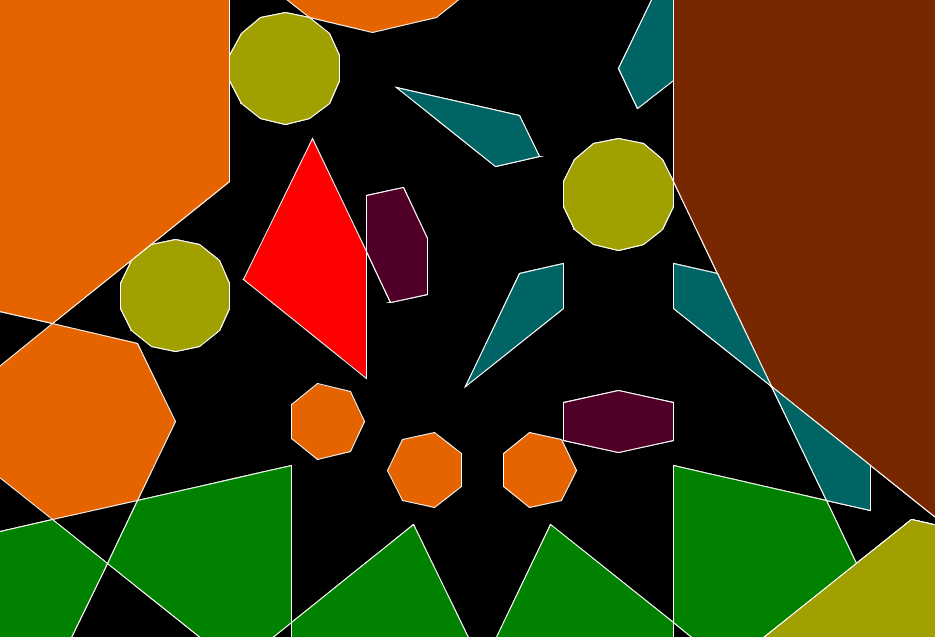}
		\par\end{centering}
	\caption{A zoo of `exotic' periodic tiles}\label{fig:septagon2}
\end{figure}

R. Schwartz (private communication) suspected that the question of existence of infinitely many non-similar periodic tiles outside the regular septagon is tied with the problem of whether arbitrarily large integers appear in the continued fractional expansion of the real part of $e^{2\pi i/7}$ (which is not known). 

Explicit connections between continued fractional expansions and renormalizations of piecewise isometric systems have been made in several situations. To describe such a connection, consider the simplest case where we have a one-parameter set of piecewise isometries $\{\Phi(\alpha), \alpha \in [0,1]\}$ (e.g. the outer billiards map of a one-parameter family of convex polygons). We are interested in cases where there exists a map $f:[0,1] \rightarrow [0,1]$ such that the dynamics of $\Phi(\alpha)$, at least locally, renormalizes to the dynamics of $\Phi(f(\alpha))$. In this situation, $\Phi(\alpha)$ will have a self-renormalization scheme if $\alpha$ is a periodic (or eventually periodic) point of $f$. Such schemes were found and analyzed in at least two different situations; the outer billiards map of `semi-regular' octagons in \cite{MR3186232}, and a certain piecewise isometries on the `square pillowcase' in \cite{MR3010377}. Remarkably, in both situations the map $f$ is closely related to the continued fraction expansion. 

In any case, it does not seem to be a coincidence that the set of regular polygons whose renormalization scheme has found is precisely the values of $n$ for which $\phi(n) \leq 2$. One may hope to find a similar renormalization scheme when all vertices lie in a single quadratic field over the rationals. Such a statement was indeed proved in one-dimensional setting (i.e., interval exchange transformations) \cite{MR1482988} with a partial converse.

\section*{Acknowledgements}

I thank my advisors Prof. Tabachnikov, Prof. Hooper, T. Aougab, and D. Davis when I was a participant of the ICERM undergraduate research program in 2012. I thank Prof. Hooper and Prof. Schwartz for providing essential insights and guidance regarding the systems. 
The computer program developed by R. Schwartz and the program developed by P. Hooper and J. Lachance were both essential to this research. The former was used to generate Figures \ref{fig:necklace}, \ref{fig:septagon}, \ref{fig:septagon2}. Figure \ref{fig:cat1} was produced by R. Schwartz and I took it from the website \url{icerm.brown.edu}.

I am supported by the Samsung Scholarship.

\appendix 
\section{Appendix}\label{appendix}

Here we prove the `full' theorem for the square, which is a strengthening of Theorem \ref{thm:sq1}. This result was proved in the author's undergraduate thesis \cite{Jeong2013} and independently in \cite{GianluigiDelMagno2013} by a different method. 

\begin{theorem}
	\label{thm:square}
	Consider the case $n=4$. For each $k \geq 1$, if we let $\lambda_k$ be the unique root of the polynomial $p_k(\l)=1-\l^{k-1}-\l^k +\l^{2k}$ in $[0,1)$, then the sequence $\l_k$ is strictly increases to 1 with $\l_1=0$. 
	Then for all $\l_k < \l \leq \l_{k+1}$, there exists exactly $k$ periodic orbits for $T_\l$ and all other orbits are attracted to one of them. 
\end{theorem}

\begin{proof}[Proof of Theorem \ref{thm:square}]
	We begin with the statements involving $\l_k$. When $k=1$, $p_{k}(\l)=-\l+\l^{2}$ and $\l_{1}=0$. Now assume $k\geq2$.
	Since $p_{k}(0)=1$ and $p_{k}(1)=0$, it is enough to prove that
	there exists a point $a_{k}$ such that $p_{k}$ is decreasing in
	the interval $(0,a_{k})$ and increasing in the interval $(a_{k},1)$.
	The derivative $p_{k}'(\lambda)$ has the form $-\lambda^{k-2}q_{k}(\lambda)$,
	so let us show $q_{k}(\lambda)$ has a unique root $b_{k}$ in the
	interval $(0,1)$ such that $q_{k}(\lambda)>0$ when $0<\lambda<b_{k}$
	and $q_{k}(\lambda)<0$ when $b_{k}<\lambda<1$. The derivative $q_{k}'(\lambda)=k-2k(k+1)\lambda^{k}$
	is monotonic and has a unique root in $(0,1)$. Since $q_{k}(0)=k-1>0$
	and $q_{k}(1)=-1<0$, we are done.
	
	To prove $\lambda_{k}<\lambda_{k+1}$, since $p_{k}(\lambda)<0$ only
	when $\lambda>\lambda_{k}$ in the interval $(0,1)$, it is enough
	to check $p_{k}(\lambda_{k+1})<0$, which is elementary. Because $\lambda_{k}<1$
	for all $k$ and $\lambda_{k}$ is increasing, the limit of $\lambda_{k}$
	exists, which we denote by $l\leq1$. From the equation 
	\[
	1=\lim_{k\rightarrow\infty}l^{k-1}(1+l-\lim_{k\rightarrow\infty}l^{k+1}),
	\]
	one sees that $l=1$.
	
	Note that in light of the baby version of the theorem, it only remains to prove that the periodic orbit for $T_\l$ of index $k$ exists precisely on the interval $(\l_k,1]$. For definiteness, assume that the square has vertices $(\pm1,\pm1)$ and set $S_k$ to be the square of index $k$ with center $(-2k,0)$. Denote the first periodic part of the code of $S_k$ (which has length $4k$) by $C_k$ and the coordinates of the corresponding hypothetical periodic point by $q_k(\l)$. 
	
	Obviously, $q_k(\l) \in \overline{S_k}$ is necessary for the $k$th periodic orbit to exist for $T_\l$. The following two lemmas will conclude the proof by showing its sufficiency. 
	
	The plan is as follows. First, we show that the condition $q_k(\l) \in \overline{S_k}$ is equivalent with $\l_k \leq \l \leq 1$. Then, we proceed to show that when $\l_k < \l$, $q_k(\l)$ indeed defines a periodic orbit for $T_\l$. Finally, we need to take care of the case $\l = \l_k$ where a `degenerate' periodic orbit exists. 
\end{proof}

\begin{lemma}\label{Coordinates}
	For $0<\lambda\leq1$, coordinates of the hypothetical periodic point $q_{k}(\lambda)=(x_k(\l),y_k(\l))$ have the following explicit formulas:
	\[
	x_k(\lambda)=-\frac{(1+\l)(1-\lambda^{2k})}{(1-\l)(1+\lambda^{2k})} \qquad y_k(\lambda)=\frac{(1+\l)(1-\lambda^{k})^{2}}{(1-\l)(1+\lambda^{2k})}.
	\]
\end{lemma}

\begin{proof}
	We will only prove the statement regarding $x_k(\l)$. For $0 \leq i \leq 4k-1$, set $v_{x,i}$ be the $x$-coordinate of the vertex $v_i$, which corresponds to the $i$th code of $q_{k}$. For $k=3$, the sequence $\{ v_{x,i} \}$ will be -1, +1, -1, +1, -1, +1, +1, -1, +1, -1, +1, -1. One can easily show that the sequence of $v_{x,i}$ alternates between -1 and +1 for $i=1,...,2k$ starting with a -1, and again alternates for $i=2k+1,...,4k$, this time starting with +1. Now from equation \ref{eq:periodic_point}, we have
	\[
	x_k(\lambda)=(\sum_{i=0}^{4k-1} \lambda^{4k-1-i}(-1)^{i-1}v_{x,i})(1+\lambda)/(1-\lambda^{4k})
	\]
	and we can simply compute the right hand side as a sum of two geometric series. This gives us the desired formula.
\end{proof}

Notice that points $q_k(\l)$, viewed as a rational curve in $\l$, converges
to the center $(-2k,0)$ as $\lambda\rightarrow1$ (See Corollary \ref{cor:lattice}).

\begin{lemma}
	\label{Finiteness Lemma}
	The point $q_{k}(\lambda)$ is a $4k$-periodic point for $T_{\lambda}$ if and only if $\l_{k} < \l \leq 1$. For $ \l_k = \l$, $q_k(\l)$ defines a degenerate $4k$-periodic point for $T_\l$ which is non-attracting.
\end{lemma}

Before we proceed to the proof, let us clarify the statement. Recall that $S$ was defined as the union of singular rays (Definition \ref{system}). By a degenerate periodic orbit (for $T_\l$), we mean a periodic orbit in which some of the points may lie on $S$. For those points, $T_\l$ is defined as one of two natural choices. To conclude the proof of the theorem, we must make sure that when $q_k(\l)$ defines a degenerate periodic orbit, no orbits are asymptotic to it.


\begin{proof}
	
	We begin with the observation that $y_k(\l)$ should be not greater than +1 for $q_k(\l)$ to define a valid periodic orbit. From the explicit formula for $y_k(\l)$, we see that $y_k(\l) \leq 1$ if and only if $\l^{2k}-\l^k+\l^{k-1}+1 \leq 0$, which happens if and only if $\l \in [\l_k, 1]$. Therefore, $\l_k \leq \l$ is necessary. Now we proceed to show the sufficiency.
	
	Fix a $\l_k \leq \l$ and for the point $H_0 := q_k(\l)$, we construct extra $4k-1$ points as follows. First, rotate $H_0$ counterclockwise with respect to the center by $\pi/2, \pi, 3\pi/2$ to obtain points $E_0, F_0,$ and $G_0$. On the segment $H_{0}E_{0}$, we pick points $E_{k-1},E_{k-2},...,E_{1}$ in a way that the lengths of segments satisfy $|\overline{E_{j+1}E_{j}}|/|\overline{E_{j}E_{j-1}}|=\lambda$ for all $1\leq j\leq k-1$ ($E_{k}=H_{0}$). Then we construct points $F_{j}$, $G_{j}$, $H_{j}$ similarly that $4k$ points have $\pi/2$-rotational symmetry with respect to the origin. 
	
	Let us show that points $E_{k-1},...,E_0$ gets reflected on the vertex $D=(+1,-1)$. For this, it is enough to show that the $x$-coordinate of $E_0$ does not exceed 1 (which is obvious as the $y$-coordinate of $H_0$ is positive) and that the $y$-coordinate of $E_{k-1}$ does not exceed -1. For the latter, from the construction we have that the $y$-coordinate of $E_{k-1}$ is given by the convex combination
	\[
	\hat{y}_k(\lambda)=\frac{1-\lambda^{k-1}}{1-\lambda^{k}}y_k(\lambda)+\frac{\lambda^{k-1}-\lambda^{k}}{1-\lambda^{k}}x_k(\lambda)
	\]
	and solving for $\hat{y}_k(\l) \leq -1$ gives again $\l^{2k}-\l^k+\l^{k-1}+1 \leq 0$. Now by symmetry, points $F_{k-1},...,F_{0}$ gets reflected on the vertex $C=(+1,+1)$, and so on. 
	
	We now claim that these $4k$ points are $T_\l$-invariant. By a direct computation using explicit formulas, one first verifies $T_\l H_0 = F_1$. Then consider two triangle $\triangle H_{0}H_{1}A$ and $\triangle F_{1}F_{2}A$. Since $H_{0},A,F_{1}$ are shown to be collinear and segments $H_{0}G_{0}$
	and $E_{0}F_{0}$ are parallel, $\angle H_{1}H_{0}A=\angle F_{2}F_{1}A$. Moreover, we have ratios $|\overline{H_{0}A}|/|\overline{F_{1}A}|=|\overline{H_{0}H_{1}}|/|\overline{F_{1}F_{2}}|=\lambda$
	from our construction. Therefore, these triangles are similar, and $T_{\lambda}H_{1}=F_{2}$. Likewise, we prove that $T_{\lambda}H_{j}=F_{j+1}$ for
	$j=1...k-1$, where $F_{k}=E_{0}$. Then we are done by rotational symmetry.
	
	Finally, we need to show that the degenerate periodic orbit formed by $q_{k}$ when $\l=\l_{k}$ is non attracting. It is enough to show
	that no points inside two lattice squares adjacent to the singular
	point $E_{k-1}$ converge to $E_{k-1}$. No points from the square
	above $E_{k-1}$ certainly cannot converge to $E_{k-1}$ since it
	reflects on $A$ in the beginning. Next, we may assume that a small
	open set inside the lattice square below $E_{k-1}$, after $2k$ iterates
	of $T_{\lambda_{k}}$, becomes an open set touching the singular point
	$G_{k-1}$ because otherwise we are done. However, since $2k$ is
	even, the latter open set lies below $G_{k-1}$ and reflects on $C$
	rather than on $B$. We are done. 
	
\end{proof}

For the equilateral triangle and the regular hexagon, we have exactly the same statements: For each of them, there exists an increasing sequence of numbers $l_1 = 0, l_2,...$ with the limit 1 and the property that for any $ l_i < \l \leq l_{i+1}$ there exists $i$ periodic orbits to which all other orbits are asymptotic. Our method of proof is expected to carry over to these cases as well.

\bibliographystyle{plain}
\bibliography{regular}

\end{document}